\documentclass{amsart}

\usepackage{amsfonts}
\usepackage{amssymb}
\usepackage{amsthm}
\usepackage[dvips]{graphicx}

\newtheorem{lemma}{Lemma}
\newtheorem{proposition}[lemma]{Proposition}

\newcommand*{\arccosh}{\mathop{\mathrm{arccosh}}\displaylimits}

\begin{document}

\title[Energy-based computational method]{An energy-based computational method in the analysis of the transmission of energy in a chain of coupled oscillators}

\author{J. E. Mac\'{\i}as-D\'{\i}az}
\address{Departamento de Matem\'{a}ticas y F\'{\i}sica, Universidad Aut\'{o}noma de Aguascalientes, Aguascalientes, Ags. 20100, Mexico}
\address{Department of Physics, University of New Orleans, New Orleans,  LA 70148}
\email{jemacias@correo.uaa.mx}

\author{A. Puri}
\address{Department of Physics, University of New Orleans, New Orleans,  LA 70148}
\email{apuri@uno.edu}


\date{\today}

\begin{abstract}
In this paper we study the phenomenon of nonlinear supratransmission in a semi-infinite discrete chain of coupled oscillators described by modified sine-Gordon equations with constant external and internal damping, and subject to harmonic external driving at the end. We develop a consistent and conditionally stable finite-difference scheme in order to analyze the effect of damping in the amount of energy injected in the chain of oscillators; numerical bifurcation analyses to determine the dependence of the amplitude at which supratransmission first occurs with respect to the frequency of the driving oscillator are carried out in order to show the consequences of damping on harmonic phonon quenching and the delay of appearance of critical amplitude. 
\end{abstract}

\maketitle

\section{Introduction}

The study of nonlinear continuous media described by modified Klein-Gordon equations subject to initial conditions is a topic of interest in several branches of the physical sciences \cite{Remoissenet,Lomdahl,Makhankov,Barone}. Here the analytical study on features of solutions of Klein-Gordon-like equations as well as the development of new and more powerful computational techniques to approximate them have been the most transited highways in the mathematical research of the area.

The behavior of nonlinear continuous media described by modified Klein-Gordon equations subject to boundary conditions is also an interesting topic of study in mathematical physics. This type of mathematical models have proved to be useful when applied, for instance, to the description of the project of data transmission in optical fibers in nonlinear Kerr regimes \cite {Hasegagwa,Mollen} or in the study of the property of self-induced transparency of a two-level system submitted to a high-energy incident (resonant) laser pulse \cite {McCall,Ablowitz}. More concretely, the behavior of a continuous medium submitted to a continuous wave radiation constitutes a fundamental problem that has not been studied in-depth, yet it has shown to have potential applications as a model in the study of Josephson transmission lines \cite{Ustinov,Ustinov2,Ustinov3}.

Numerical studies on discrete versions of these models \cite{frenchpaper} followed by experimental results \cite{Geniet2} have shown that there exists a bifurcation of wave transmission within a forbidden band gap in certain semi-infinite undamped Klein-Gordon-like chains of coupled oscillators which are periodically forced at the end. This bifurcation is a consequence of the generation of nonlinear modes by the periodic forcing at the end of the chain, and allows energy to flow into the medium via a nonlinear process called \emph {nonlinear supratransmission}, which has been proved numerically not to depend on integrability as long as the model possesses a natural forbidden band gap. 

The process of nonlinear supratransmission has been studied numerically by means of computational algorithms already built in standard mathematical packages. Most particularly, the numerical results obtained in \cite{frenchpaper}, for instance, rely on the use of a Runge-Kutta method of high order, which has the advantage of possessing high accuracy and efficiency in the computations, but lacks the consistency in the numerical estimation of the energy flowing into the medium which is desired in the study of supratransmission. With this shortcoming in mind, we present in this paper an alternate numerical formulation to study the process of supratransmission in a nonlinear system of differential equations, that has the advantage of being consistent in the approximation of the solutions to the problem and in the estimation of the continuous energy. More concretely, in the present work we study the process of nonlinear supratransmission in a semi-infinite nonlinear discrete system of coupled oscillators governed by modified sine-Gordon and Klein-Gordon equations with constant internal and external damping. We develop some numerical results for dissipative nonlinear modified Klein-Gordon-like equations, and validate our conclusions against \cite{frenchpaper}. Our study --- rigorous in numerical nature --- will soon be succeeded by future applications in forthcoming papers.

In Section II of this paper we introduce the mathematical problem under study and derive the expression of the instantaneous rate of change of the energy transmitted to the system through the boundary. Section III is devoted to introducing the finite-difference scheme; here we present the discrete energy analysis associated with the problem under study and establish in detail the numerical properties of our method. Numerical results are presented in the next section, followed by a brief discussion.

\section{Analytical results}

\subsection{Mathematical model}

In this article we study the effects of the nonnegative constant parameters $\beta$ and $\gamma$, and the real constant $m ^2$ on the behavior of solutions to the mixed-value problem with mass $m$,
\begin{equation}
\begin{array}{c}
\begin{array}{rcl}
\displaystyle {\frac {d ^2 u _n} {d t ^2} - \left( c ^2 +\beta \frac {d} {d t} \right) \Delta ^2 _x u _n + \gamma \frac {d u _n} {d t} + m ^2 u _n + V ^\prime (u _n)} & = & 0,
\end{array}\\ \\
		\begin{array}{rl}
        \begin{array}{l}
            {\rm subject\ to:} \\ \\ \\
        \end{array}
        \left\{
        \begin{array}{ll}
            u _n (0) = \phi (n), & n \in \mathbb {Z} ^+, \\
            \displaystyle {\frac {d u _n} {d t} (0) = \varphi (n)}, & n \in \mathbb {Z} ^+, \\
            u _0 (t) = \psi (t), & t \geq 0,
        \end{array}\right.
    \end{array}
\end{array}\label{Eqn:DiscreteMain}
\end{equation}
which describes a system of coupled oscillators with coupling coefficient $c \gg 1$, and where $\beta$ and $\gamma$ clearly play the roles of internal and external damping coefficients, respectively. The number $\Delta ^2 _x u _n$ is used to represent the spatial second-difference $u _{n + 1} - 2 u _n + u _{n - 1}$ for every $n \in \mathbb {Z} ^+$, and the boundary-driving function $\psi$ is assumed continuously differentiable on $(0 , + \infty)$. In our study, we will let $V (u) = 1 - \cos u$ for a chain of coupled {\it sine-Gordon} equations, and $V (u _n) = \frac {1} {2!} u _n ^2 - \frac {1} {4!} u _n ^4 + \frac {1} {6!} u _n ^6$ for zero mass in the case of a {\it Klein-Gordon} chain.

It is worth recalling that the differential equation under study has multiple applications in the continuous limit. For instance, a similar equation appears in the study of Josephson junctions between superconductors when dissipative effects are taken into account \cite {Remoissenet}. A modification of this equation is used in the study of fluxons in Josephson transmission lines \cite {Lomdahl}, and a version in the form of a modified Klein-Gordon equation appears in the statistical mechanics of nonlinear coherent structures such as solitary waves (see \cite {Makhankov} pp. 298--309) when no internal damping is present. Besides, our equation clearly describes the motion of a damped string in a non-Hookean medium.

For purposes of this paper, we will consider a system of coupled oscillators initially at rest at the origin of the system of reference, subject to an external harmonic forcing described by $\psi (t) = A \sin \Omega t$, and pure-imaginary or real mass satisfying $| m | \ll 1$. Analysis of the undamped linearized system of differential equations in (\ref {Eqn:DiscreteMain}) shows that the linear dispersion relation reads $\omega ^2 (k) = m ^2 + 1 + 2 c ^2 (1 - \cos k)$
in any case. The driving frequency will take on values in the forbidden band gap region $\Omega < \sqrt {m ^2 + 1}$, in which case the linear analysis yields the exact solutions $u _n (t) = A \sin (\Omega t) e ^{- \lambda n}$, where
$$
\lambda = \arccosh \left( 1 + \frac {m ^2 + 1 - \Omega ^2} {2 c ^2} \right).
$$

Meanwhile, the massless undamped nonlinear case possesses an exact solution in the continuous limit which happens to work well for high values of the coupling coefficient. It has been shown numerically that, for each frequency $\Omega$ in the forbidden band gap, the massless medium starts to transmit energy by means of nonlinear mode generation for amplitudes greater than the threshold value $A _s = 4 \arctan \left( {\lambda c} / {\Omega} \right)$.

\subsection{Energy analysis}

In this section we derive the expression of the instantaneous rate of change of total energy in system (\ref {Eqn:DiscreteMain}). Here, by a {\it square-summable} sequence we understand a real sequence $( a _n ) _{n = 0} ^\infty$ for which $\sum a _n ^2$ is convergent.

\begin{lemma}[Discrete Green's First Identity]
If $(a _n) _{n = 0} ^\infty$ is a square-summable sequence then
$$
\sum _{n = 1} ^\infty (a _{n + 1} - 2 a _n + a _{n - 1}) a _n = a _0 (a _0 - a _1) - \sum _{n = 1} ^\infty (a _n - a _{n - 1}) ^2.
$$
\end{lemma}

\begin{proof}
H\"{o}lder's inequality implies that both series in the formula of the lemma converge. Moreover, the sequence defined by $t _n = a _n a _{n - 1} - (a _{n - 1}) ^2$ for every positive integer $n$ converges to zero and the result follows now from the identities
\begin{eqnarray*}
\sum _{n = 1} ^\infty (a _{n + 1} - 2 a _n + a _{n - 1}) a _n & = & \sum _{n = 1} ^\infty (t _{n + 1} - t _n) - \sum _{n = 1} ^\infty (a _n - a _{n - 1}) ^2 \\
 & = & - t _1 - \sum _{n = 1} ^\infty (a _n - a _{n - 1}) ^2.
\end{eqnarray*}
\end{proof}

\begin{proposition}
Let $( u _n (t) ) _{n = 0} ^\infty$ be solutions to {\rm (\ref {Eqn:DiscreteMain})} such that $( \dot {u} _n (t) ) _{n = 0} ^\infty$ is square-summable at any fixed time $t$. The instantaneous rate of change of the total energy in the system is given by
$$
\frac {d E} {d t} = c ^2 (u _0 - u _1) \dot {u} _0 - \gamma \sum _{n = 1} ^\infty (\dot {u} _n) ^2 - \beta \left[ \sum _{n = 1} ^\infty (\dot {u} _n - \dot {u} _{n - 1}) ^2 + (\dot {u} _1 - \dot {u} _0) \dot {u} _0 \right].
$$
\end{proposition}

\begin{proof}
The energy density of the undamped system of coupled equations with a potential energy $V (u_n)$ in the $n$-th oscillator is given by $H _n = \frac {1} {2} [\dot {u} _n ^2 + c ^2 (u _{n + 1} - u _n) ^2 + m ^2 u _n ^2] + V (u _n)$. After including the potential energy from the coupling between the first two oscillators, the total energy of the system at any time becomes
$$
E = \sum _{n = 1} ^\infty H _n + \frac {c ^2} {2} (u _1 - u _0) ^2,
$$
and the fact that $u _n$ tends to $0$ as $n$ increases implies that the sequence $J _n = - c ^2 \dot {u} _n (u _n - u _{n - 1})$ converges to zero pointwisely at any fixed time. Simplifying and rearranging terms of the derivative of the Hamiltonian $H _n$ yields the expression
$$
\frac {d H _n} {d t} = (J _n - J _{n + 1}) + \beta (\dot {u} _{n + 1} - 2 \dot {u} _n + \dot {u} _{n - 1}) \dot {u} _n - \gamma (\dot {u} _n) ^2.
$$
The result follows now from this identity after computing the derivative of the total energy of the system, using the formula for telescoping series and applying our discrete version of Green's First Identity.
\end{proof}

Observe from the proposition that the expression of the rate of change of energy associated with damped system (\ref {Eqn:DiscreteMain}) is in agreement with the undamped formula derived in \cite {frenchpaper}. Moreover, it is clear that the external damping coefficient contributes decreasing the total amount of energy in the chain system for $\beta$ equal to zero.

\section{Numerical analysis}

\subsection{Finite-difference scheme}

From a practical point of view, we will restrict our study to a system consisting of a finite number $N$ of coupled oscillators with constant internal and external damping coefficients, described by the system of ordinary differential equations
$$
\displaystyle {\frac {d ^2 u _n} {d t ^2} - \left( c ^2 + \beta \frac {d} {d t} \right) \Delta _x ^2 u _n + \alpha \frac {d u _n} {d t} + m ^2 u _n + V ^\prime (u _n)} = 0
$$
for $1 \leq n \leq N$, where $\alpha$ includes both the effect of external damping and a simulation of an absorbing boundary slowly increasing in magnitude on the last $N - N _0$ oscillators. More concretely, we let $u _{N + 1} (t)$ be equal to zero at all time $t$, and let $\alpha$ be the sum of external damping and the function 
$$
\gamma ^\prime (n) = \left\{ \begin{array}{ll}
\kappa \left[ 1 + \tanh \left( \displaystyle {\frac {2 n - N _0 + N} {2 \sigma}} \right) \right], & N _0 < n \leq N, \\
0, & \mathrm {otherwise}.
\end{array} \right.
$$
In practice, we will let $\kappa = 0.5$, $\sigma = 3$, $N _0 = 150$ and $N = 200$. 

We proceed now to discretize problem (\ref{Eqn:DiscreteMain}) using a finite system of $N$ differential equations and a regular partition $0 = t _0 < t _1 < \dots < t _M = T$ of the time interval $[0 , T]$ with time step equal to $\Delta t$. For each $k = 0 , 1 , \dots , M$, let us represent the approximate solution to our problem on the $n$-th oscillator at time $t _k$ by $u ^k _n$. If we convey that $\delta _t u _n ^k = u _n ^{k + 1} - u _n ^{k - 1}$, that $\delta ^2 _t u _n ^k = u _n ^{k + 1} - 2 u _n ^k + u _n ^{k - 1}$ and that $\delta ^2 _x u _n ^k = u _{n + 1} ^k - 2 u _n ^k + u _{n - 1} ^k$, our problem takes then the discrete form
\begin{equation}
\begin{array}{c}
\begin{array}{rcl}
\displaystyle {\frac {\delta ^2 _t u _n ^k} {(\Delta t) ^2} - \left( c ^2 + \frac {\beta} {2 \Delta t} \delta _t \right) \delta ^2 _x u _n ^k + \frac {\alpha} {2 \Delta t} \delta _t u _n ^k +} \qquad & & \\
\displaystyle {\frac {m ^2} {2} [u _n ^{k + 1} + u _n ^{k - 1}] + \frac {V (u _n ^{k + 1}) - V (u _n ^{k - 1})} {u _n ^{k + 1} - u _n ^{k - 1}}} & = & 0,
\end{array} \\ \\
		\begin{array}{rl}
        \begin{array}{l}
            {\rm subject\ to:} \\ \\ \\
        \end{array}
        \left\{
        \begin{array}{ll}
            u _n ^0 = \phi (n), & 1 \leq n \leq N, \\
            u _n ^1 = \phi (n) + \Delta t \varphi (n), & 1 \leq n \leq N, \\
            u _0 ^k = \psi (k \Delta t), & 1 \leq k \leq M, \\
            u _{N + 1} ^k = 0, & 1 \leq k \leq M.
        \end{array}\right.
    \end{array}
\end{array}\label{Eqn:DiscreteMainDiscr}
\end{equation}

Observe that the proposed numerical method is nonlinear and requires an application of Newton's method for systems of equations in order to be implemented. Notice also that if $V ^\prime (u _n)$ at the $k$-th time step is approximated by $V ^\prime (u _n ^k)$ then the finite-difference scheme becomes linear and an application of Crout's reduction technique for tridiagonal systems suffices to approximate solutions of (\ref {Eqn:DiscreteMain}). In such case, it is readily seen that the vector equation $A \mathbf {u} ^{k + 1} = B \mathbf {u} ^k + C \mathbf {u} ^{k - 1} - V ^\prime (\mathbf {u} ^k) + \mathbf {v} ^k$ must be satisfied for every $k \geq 2$, for the three $N$-dimensional tridiagonal matrices and the $N$-dimensional vector
$$
\begin{array}{rcl}
A = \left(
\begin{array}{cccc}
b & a & \cdots & 0 \\
a & b & \cdots & 0 \\
\vdots & \vdots & \ddots & \vdots \\
0 & 0 & \cdots & b
\end{array}
\right), & \qquad &
B = \left(
\begin{array}{cccc}
d & c ^2 & \cdots & 0 \\
c ^2 & d & \cdots & 0 \\
\vdots & \vdots & \ddots & \vdots \\
0 & 0 & \cdots & d
\end{array}
\right), \\ \\
C = \left(
\begin{array}{ccccc}
e & a & \cdots & 0 \\
a & e & \cdots & 0 \\
\vdots & \vdots & \ddots & \vdots \\
0 & 0 & \cdots & e
\end{array} \right),
& &
\mathbf {v} ^k = \left(
\begin{array}{c}
c ^2 u _0 ^k - a \delta _t u _0 ^k \\
0 \\
0 \\
\vdots \\
0
\end{array}
\right),
\end{array}
$$
respectively, and constants 
$$
\begin{array}{lcr}
\displaystyle {a = - \frac {\beta} {2 \Delta t}}, & \qquad\qquad & \displaystyle {b = \frac {\alpha + 2 \beta} {2 \Delta t} + \frac {m ^2} {2} + \frac {1} {(\Delta t) ^2}}, \\
\displaystyle {d = \frac {2} {(\Delta t) ^2} - 2 c ^2} & \mathrm {and} & \displaystyle {e = \frac {\alpha + 2 \beta} {2 \Delta t} - \frac {m ^2} {2} - \frac {1} {(\Delta t) ^2}}.
\end{array}
$$
Here $\mathbf {u} ^k = (u _1 ^k , \dots , u _n ^k) ^t$ for every $k \in \{ 0 , 1 , \dots , M \}$, and $V (\mathbf {u} ^k)$ is the $n$-th dimensional vector whose $i$-th component is equal to $V (u_i ^k)$. This latter formulation of our problem will be used for validation purposes only, since we will prefer the nonlinear formulation due to the quadratic order of convergence of Newton's method and other reasons that will be presented in the next section. 

\begin{figure*}[tcb]
\centerline{
\begin{tabular}{cc}
\scriptsize{{\bf (a)} $A = 1.77$}&\scriptsize{{\bf (b)} $A = 1.79$} \\
\includegraphics[width=0.45\textwidth]{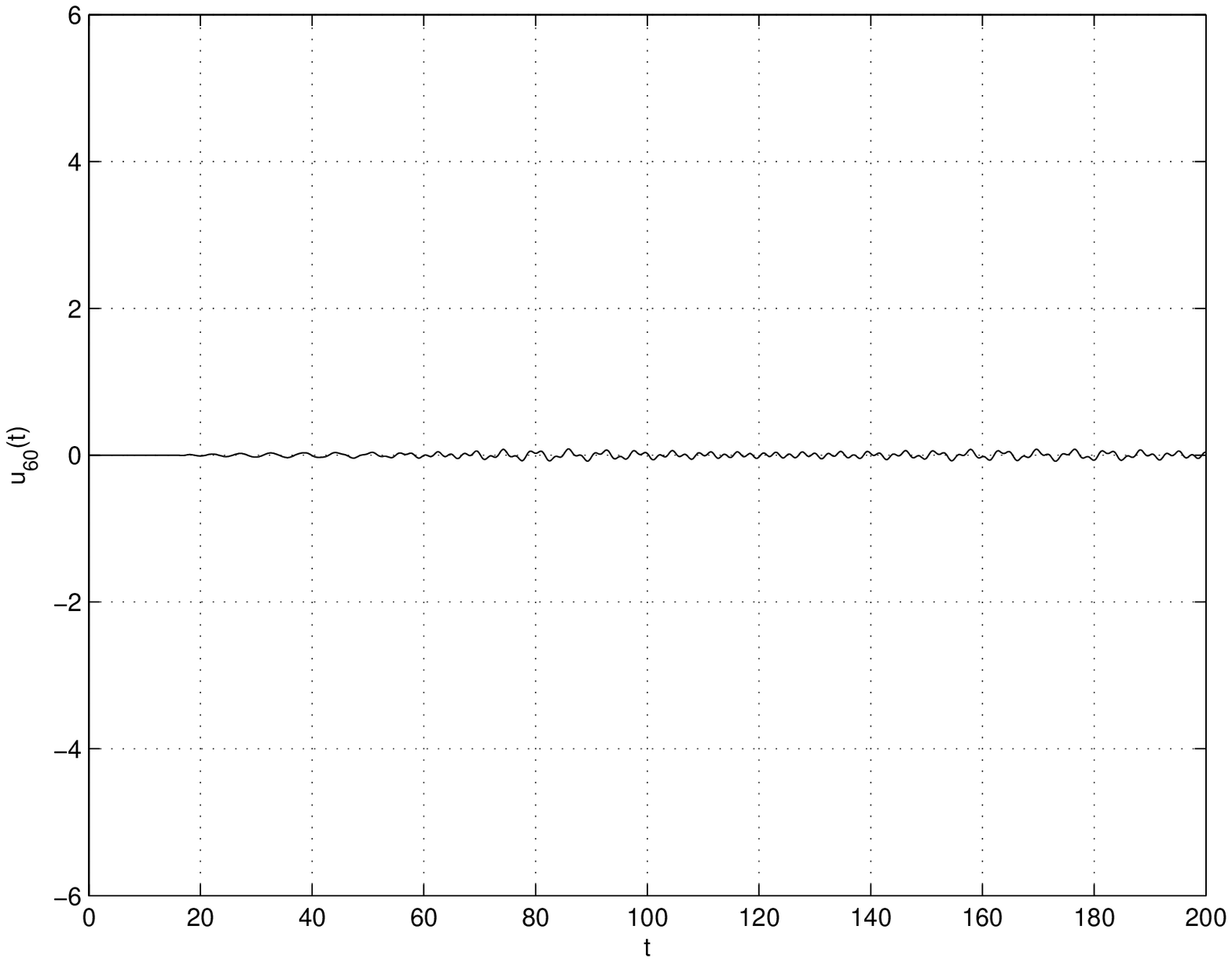}&\includegraphics[width=0.45\textwidth]{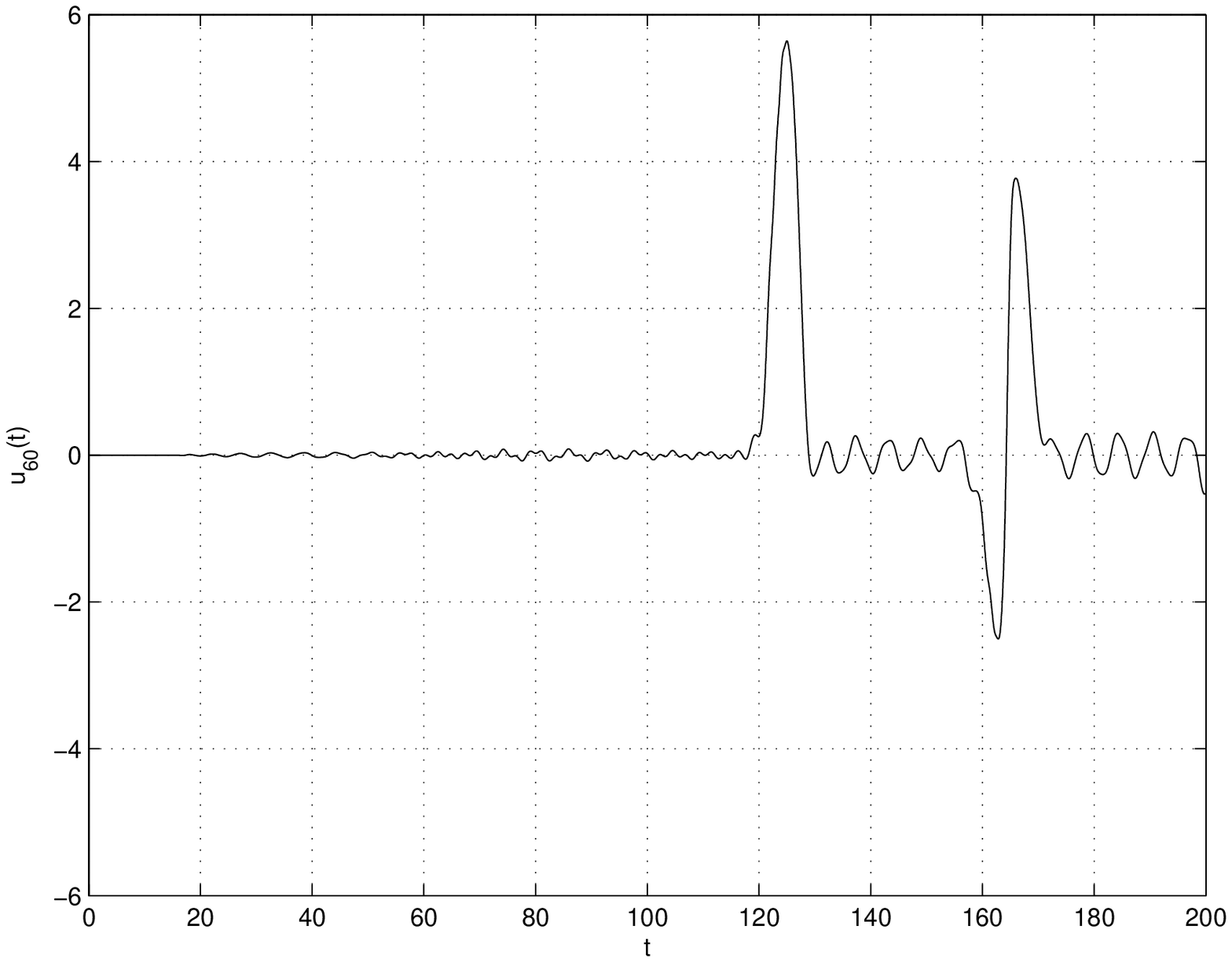} \\
\scriptsize{{\bf (a)} $A = 1.77$}&\scriptsize{{\bf (b)} $A = 1.79$} \\
\includegraphics[width=0.45\textwidth]{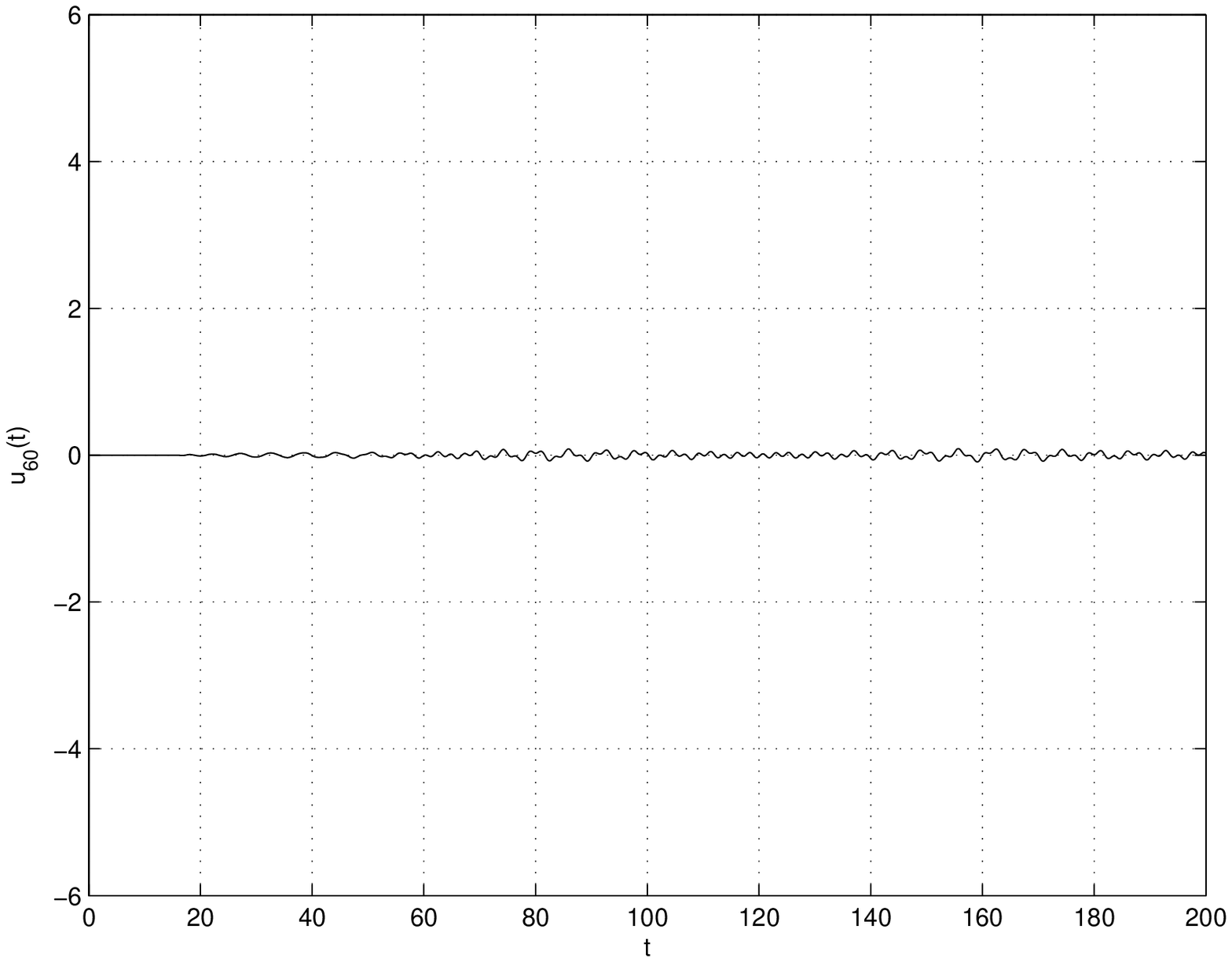}&\includegraphics[width=0.45\textwidth]{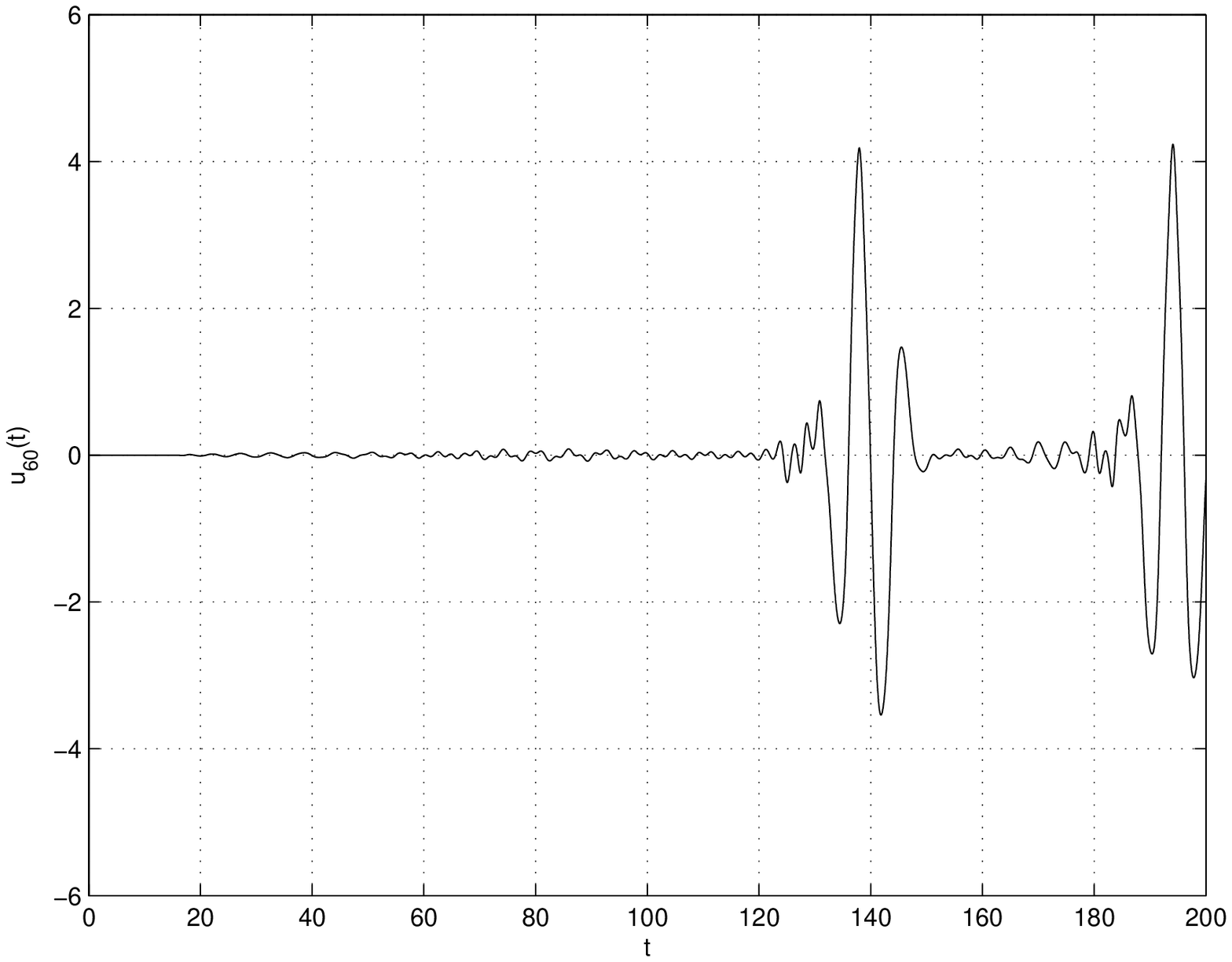}
\end{tabular}}
\caption{Approximate solution $u _{60} (t)$ of a sine-Gordon system in the first row and a Klein-Gordon system in the second, for a driving frequency of $0.9$ and two different amplitude values $A$: (a) before and (b) after the bifurcation threshold.\label{Paper2Fig1}} %
\end{figure*}

\subsection{Discrete energy}

The total energy in the system at the $k$-th time step will be approximated numerically via
\begin{eqnarray*}
E _k & = & \frac {1} {2} \sum _{n = 1} ^ M \left( \frac {u _n ^{k + 1} - u _n ^k} {\Delta t} \right) ^2 + \frac {c ^2} {2} \sum _{n = 1} ^M (u _{n + 1} ^{k + 1} - u _n ^{k + 1}) (u _{n + 1} ^k - u _n ^k) \\
 & & \qquad + \frac {m ^2} {2} \sum _{n = 1} ^M \frac {(u _n ^{k + 1}) ^2 + (u _n ^k) ^2} {2} + \sum _{n = 1} ^M \frac {V (u _n ^{k + 1}) + V (u _n ^k)} {2} \\
 & & \qquad + \frac {c ^2} {2} (u _1 ^{k + 1} - u _0 ^{k + 1}) (u _1 ^k - u _0 ^k).
\end{eqnarray*}

\begin{proposition}
The following identity holds for every sequence $(u _n ^k)$ satisfying scheme {\rm (\ref {Eqn:DiscreteMainDiscr})} and every positive index $k$:
\begin{eqnarray*}
\frac {E _k  - E _{k - 1}} {\Delta t} & = & c ^2 (u _0 ^k - u _1 ^k) \frac {\delta _t u _0 ^k} {2 \Delta t} - \gamma \sum _{n = 1} ^\infty \left( \frac {\delta _t u _n ^k} {2 \Delta t} \right) ^2 \\
 & & \qquad - \beta \left[ \sum _{n = 1} ^\infty \left( \frac {\delta _t u _n ^k - \delta _t u _{n - 1} ^k} {2 \Delta t} \right) ^2 + \left( \frac {\delta _t u _1 ^k - \delta _t u _0 ^k} {2 \Delta t} \right) \frac {\delta _t u _0 ^k} {2 \Delta t} \right].
\end{eqnarray*}
\end{proposition}

\begin{proof}
The proof of this result is merely an algebraic task. We need only say that an application of the discrete Green's First Identity with $a _n = (u _n ^{k + 1} - 2 u _n ^{k - 1}) / 2 \Delta t$ is indispensable in order to reach the correct expression of the term with coefficient $\beta$.
\end{proof}

We conclude that the discrete energy associated with scheme (\ref {Eqn:DiscreteMainDiscr}) is a consistent approximation of order $\mathcal {O} (\Delta t)$ of the total energy of system (\ref {Eqn:DiscreteMain}), whereas the discrete rate of change of energy is consistent order $\mathcal {O} (\Delta t) ^2$ with the instantaneous rate of change.

\subsection{Stability analysis}

\begin{figure*}[tbc]
\centerline{
\begin{tabular}{cc}
\scriptsize{Sine-Gordon chain}&
\scriptsize{Klein-Gordon chain} \\
\includegraphics[width=0.45\textwidth]{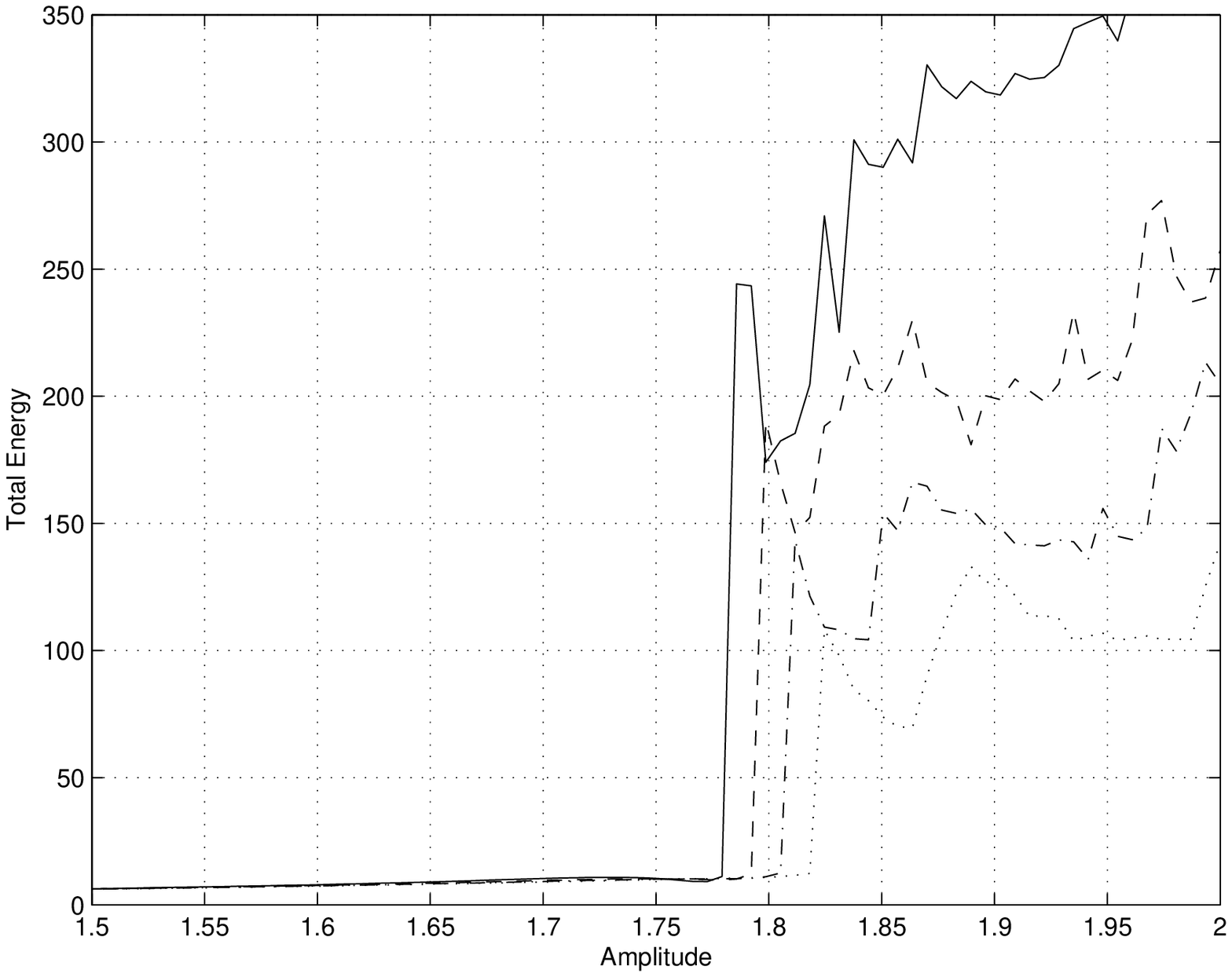}
&\includegraphics[width=0.45\textwidth]{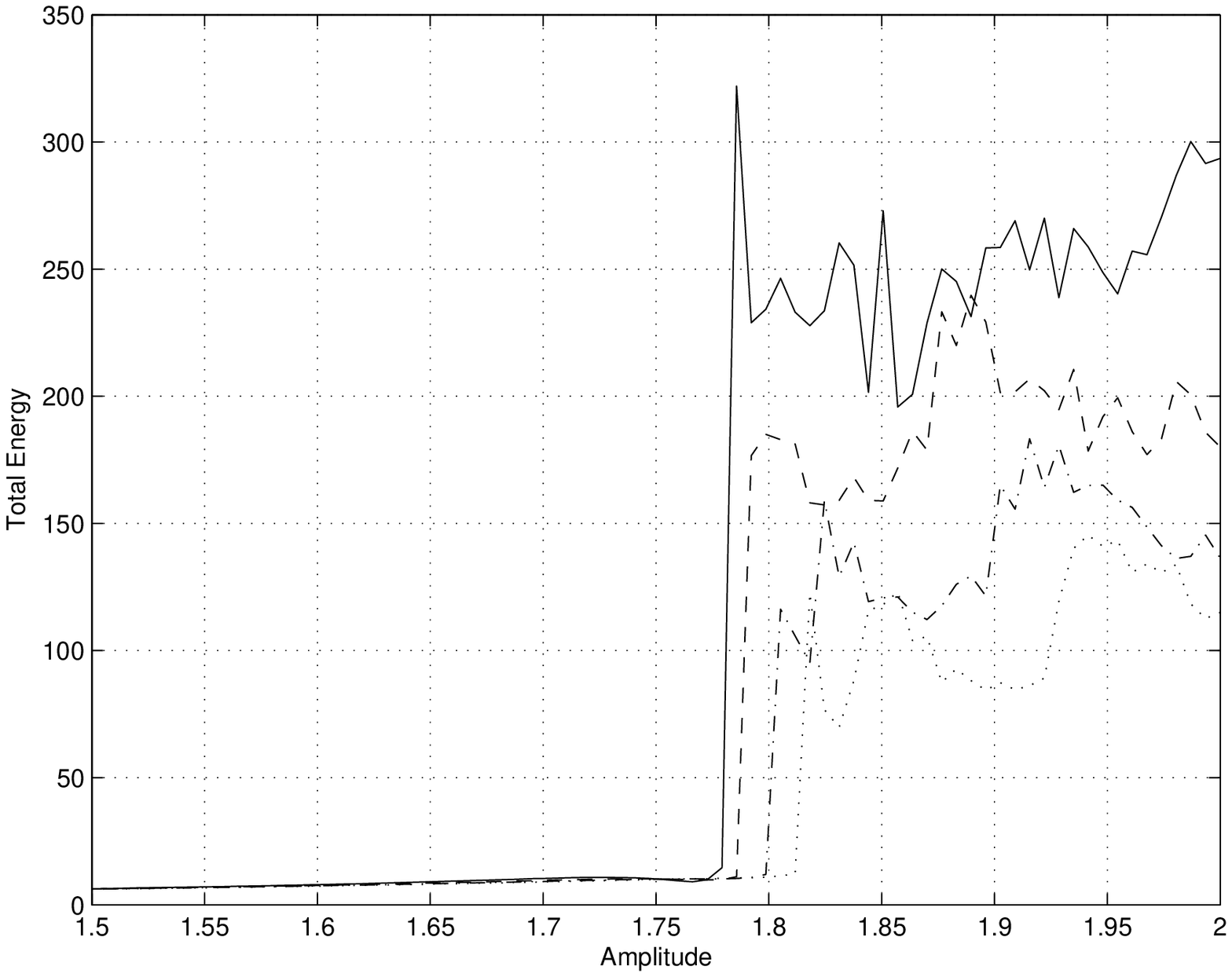} \\
\end{tabular}}
\caption{Graphs of total energy transmitted into the massless sine-Gordon and Klein-Gordon systems vs. driving amplitude for a driving frequency of $0.9$, with $\beta = 0$ and $\gamma = 0$ (solid), $0.01$ (dashed), $0.02$ (dash-dotted), $0.03$ (dotted). \label{Paper2Fig2}} %
\end{figure*}

The following result summarizes the main numerical properties of our method.

\begin{proposition}
Scheme {\rm (\ref{Eqn:DiscreteMainDiscr})} is consistent order $\mathcal {O} (\Delta t) ^2$ with the linear contribution of {\rm (\ref{Eqn:DiscreteMain})} for a potential equal to zero. Moreover, a necessary condition for the scheme to be stable order $n$ is that
$$
\left( c ^2 - \frac {m ^2} {4} \right) \left(\Delta t \right) ^2 < 1 + \left( \frac {\alpha} {4} + \beta \right) \Delta t. 
$$
\end{proposition}

\begin{proof}
Following the notation in \cite{Thomas}, let $U _{1n} ^{k + 1} = u _n ^{k + 1}$ and $U _{2n} ^{k + 1} = u _n ^k$ for each $n = 0 , 1 , \dots , M$ and $k = 0, 1, \dots, N - 1$. For every $n = 0, 1, \dots , M$ and $k = 1, 2, \dots , N$ let $\bar {U} _n ^k$ be the column vector whose components are $U _{1n} ^k$ and $U _{2n} ^k$. Our problem can be written then in matrix form as
$$
\left( \begin{array}{cc} %
g & 0 \\
0 & 1 \end{array} \right)
\bar {U} _n ^{k + 1} = \left( %
\begin{array}{cc}
2 + c ^2 (\Delta t) ^2 \delta _x ^2 & - h \\
1 & 0
\end{array}
\right) \bar {U} _n ^k, \nonumber
$$
where
\begin{eqnarray}
g & = & 1 + \alpha \frac {\Delta t} {2} - \beta \frac {\Delta t} {2} \delta _x ^2 + m ^2 \frac {(\Delta t) ^2} {2} \qquad \qquad {\rm and} \nonumber \\
h & = & 1 - \alpha \frac {\Delta t} {2} + \beta \frac {\Delta t} {2} \delta _x ^2 + m ^2 \frac {(\Delta t) ^2} {2}. \nonumber
\end{eqnarray}
Applying Fourier transform to the vector equation we obtain
$$
\hat {U} _n ^{k + 1} = \left( %
\begin{array}{cc}
\frac {2} {\hat {g} ( \xi )} \left(1 - 2 c ^2 (\Delta t) ^2 \sin ^2 \frac {\xi}
{2} \right) & - \frac {\hat {h} ( \xi )} {\hat {g} ( \xi )} \\
1 & 0
\end{array}
\right) \hat {U} _n ^k. \nonumber
$$

\begin{figure*}[tbc]
\centerline{
\begin{tabular}{cc}
\scriptsize{Pure-imaginary masses} &
\scriptsize{Real masses} \\
\includegraphics[width=0.45\textwidth]{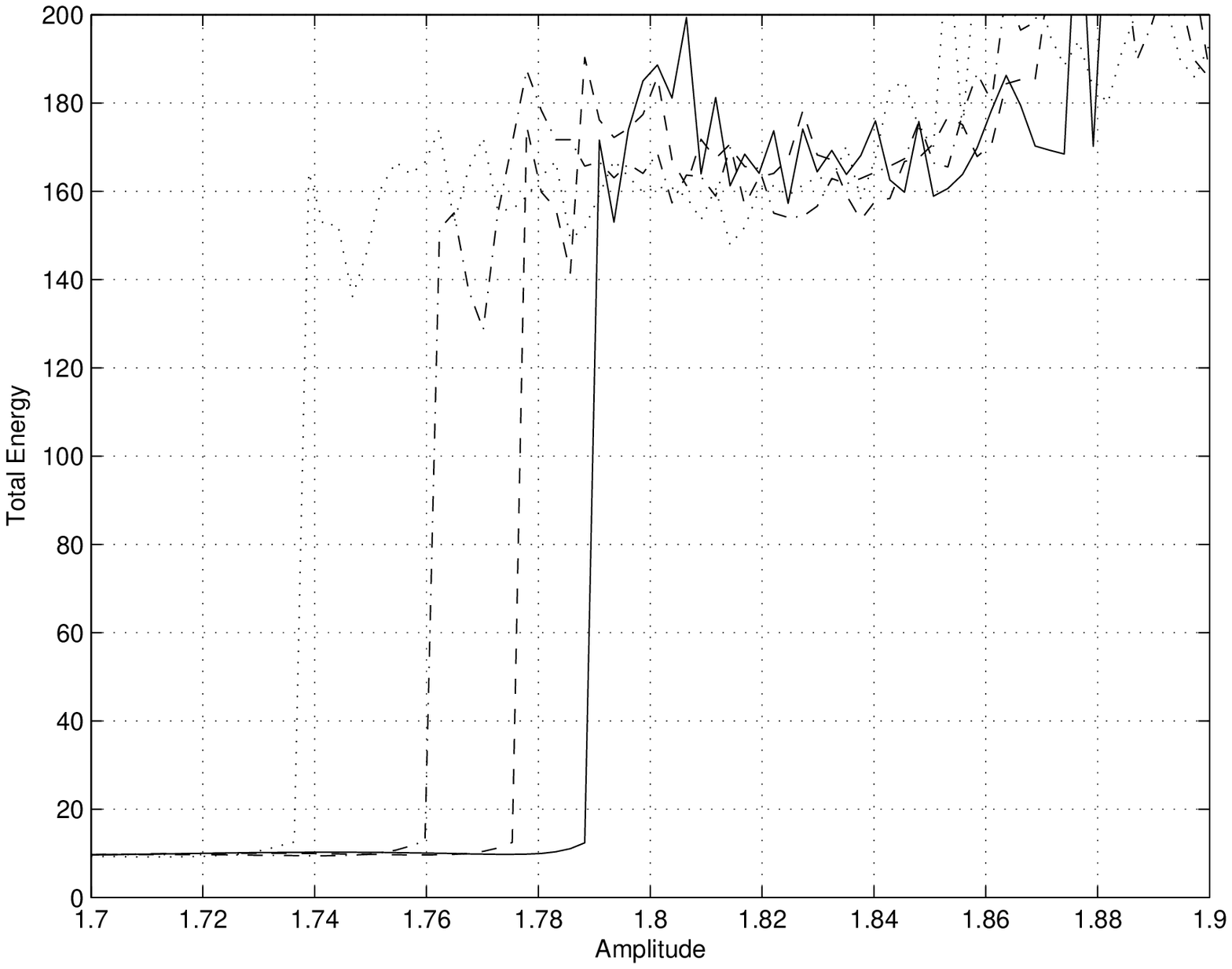} &
\includegraphics[width=0.45\textwidth]{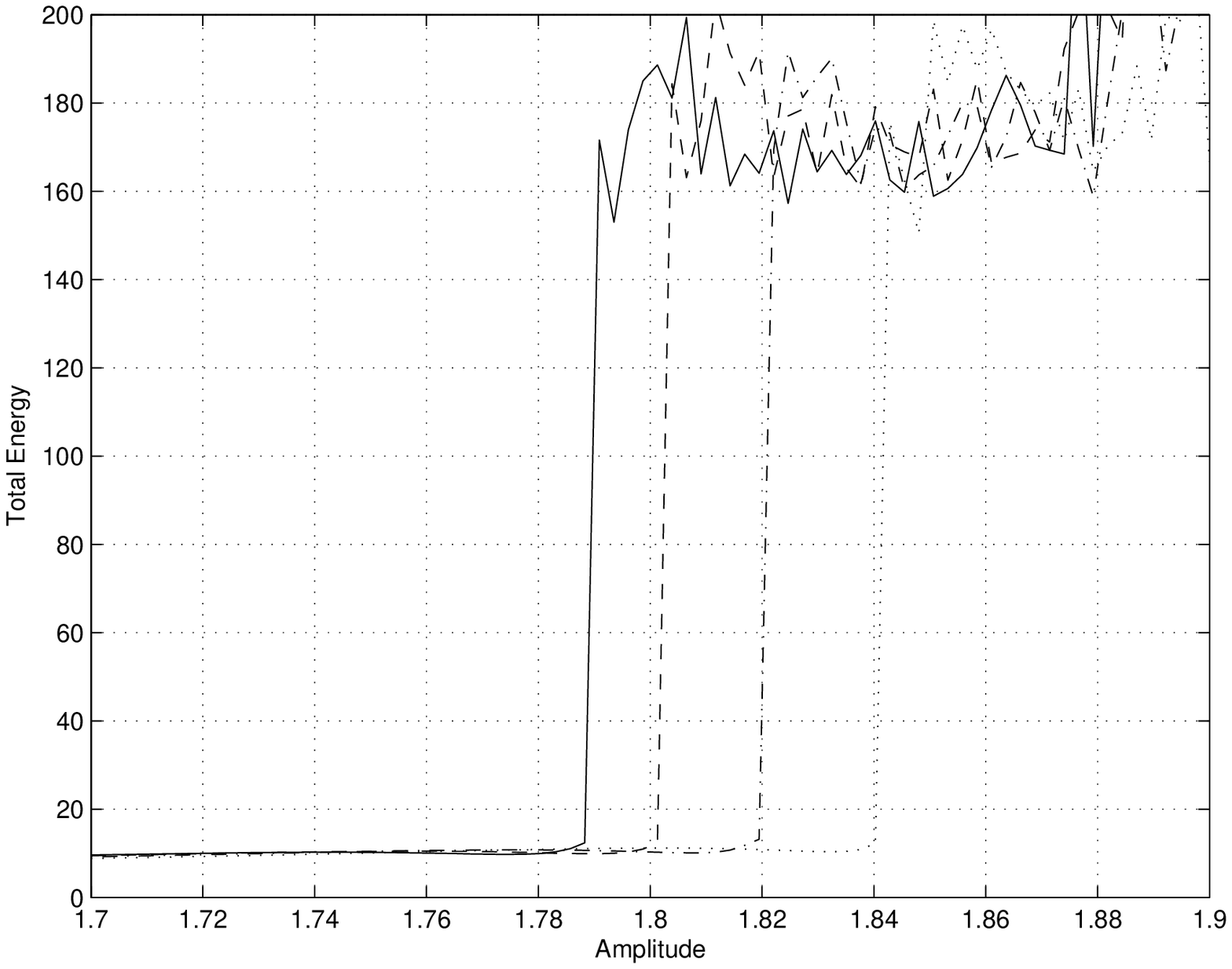} 
\end{tabular}}
\caption{Total energy transmitted into the sine-Gordon system vs. driving amplitude for a driving frequency of $0.9$, with $\beta = 0$, $\gamma = 0.01$, and pure-imaginary and real masses of magnitude $0$ (solid), $0.05$ (dashed), $0.075$ (dash-dotted) and $0.1$ (dotted). \label{Paper2Fig3}} %
\end{figure*}

The matrix $A ( \xi )$ multiplying $\hat {U} _n ^k$ in this equation is the amplification matrix of the problem. It is easy to check that the eigenvalues of $A$ for $\xi = \pi$ are given by 
$$
\lambda _\pm = \frac {1 - 2 c ^2 (\Delta t) ^2 \pm \sqrt{ (1 - 2 c ^2 (\Delta t) ^2) ^2 -
\hat {h} (\pi) \hat {g} (\pi)}} {\hat {g} (\pi)}. \nonumber
$$

Suppose for a moment that $1 - 2 c ^2 (\Delta t) ^2 < - \hat {g} (\pi)$. If the radical in the expression for the eigenvalues of $A (\pi)$ is a pure real number then $| \lambda _- | > 1$. So for every $n \in \mathbb {N}$, $|| A ^n || \geq | \lambda _- | ^n$ grows faster than $K _1 + n K _2$ for any constants $K _1$ and $K _2$. A similar situation happens when the radical is a pure imaginary number, except that in this case $| \cdot |$ represents the usual Euclidean norm in the field of complex numbers. Therefore in order for our numeric method to be stable order $n$ it is necessary that $1 - 2 c ^2 (\Delta t) ^2 > - \hat {g} (\pi)$, which is what we wished to establish.
\end{proof}

\begin{figure*}[tbct]
\centerline{
\begin{tabular}{cc}
\scriptsize{Sine-Gordon chain} &
\scriptsize{Klein-Gordon chain} \\
\includegraphics[width=0.45\textwidth]{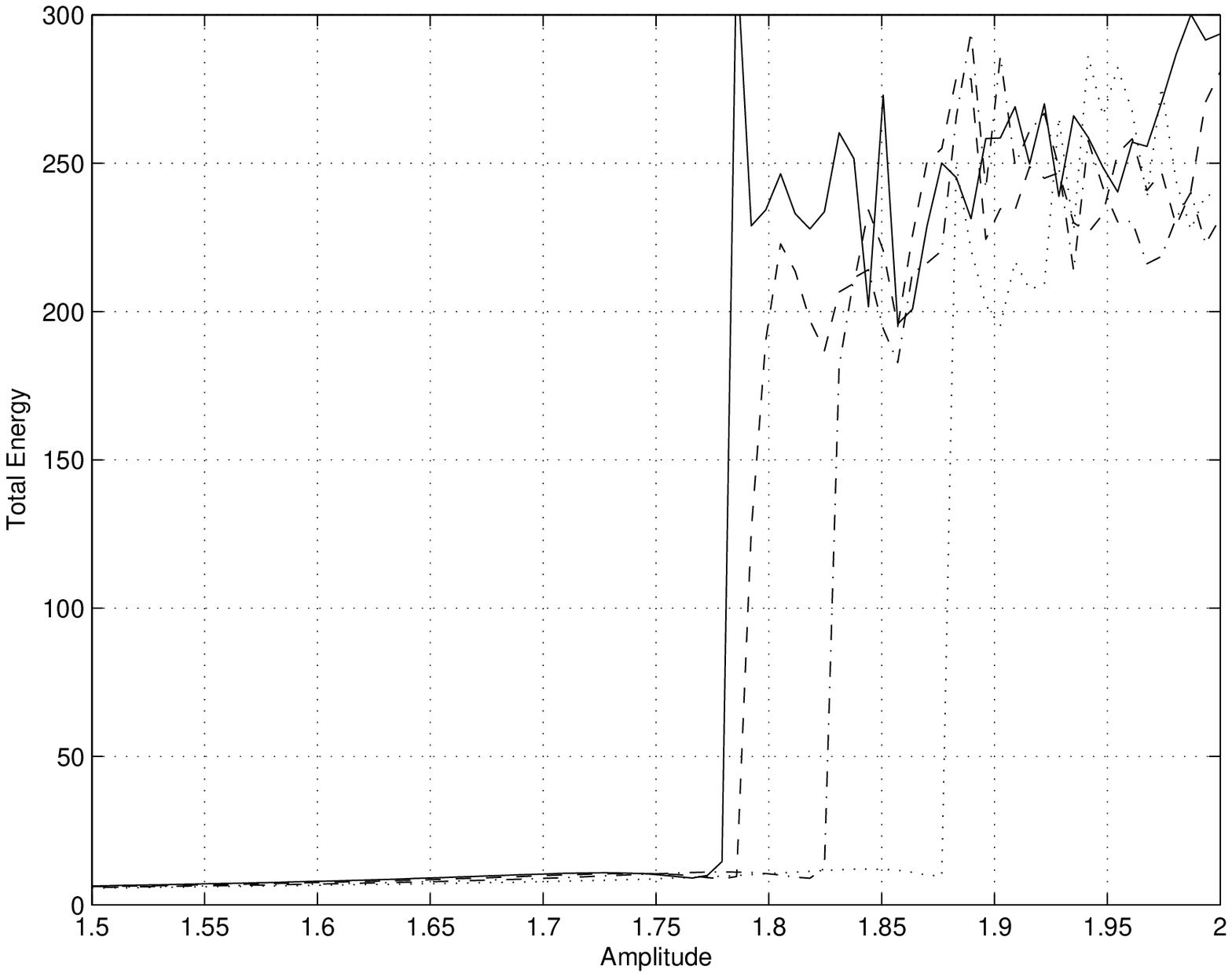} &
\includegraphics[width=0.45\textwidth]{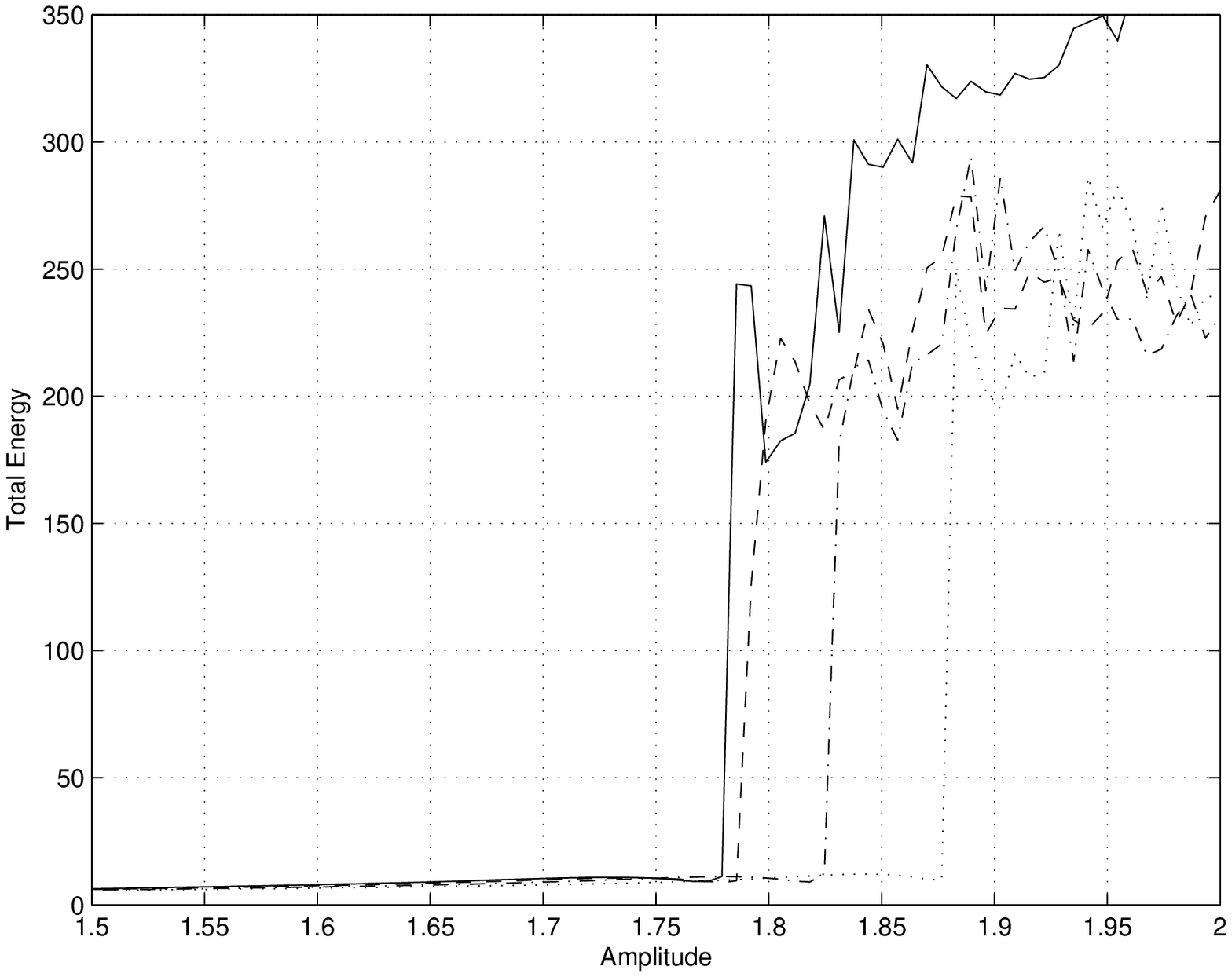}
\end{tabular}}
\caption{Graphs of total energy transmitted into the massless sine-Gordon and Klein-Gordon systems vs. driving amplitude for a driving frequency of $0.9$, with $\gamma = 0$ and $\beta = 0$ (solid), $0.1$ (dashed), $0.2$ (dash-dotted), $0.3$ (dotted). \label{Paper2Fig3half}} %
\end{figure*}

\section{Numerical results}

In this section we study the effects of the internal and external damping coefficients on the behavior of solutions to mixed-value problem (\ref{Eqn:DiscreteMain}). Particularly, we wish to establish the effect of weak damping on the minimum amplitude value necessary for the phenomenon of supratransmission to take place at a fixed driving frequency. Throughout we validate our code against \cite{frenchpaper} and against an implementation of our finite-difference scheme using the Runge-Kutta method of order four.

\subsection{External Damping}

For the remainder of this work we consider a semi-infinite coupled chain of oscillators initially at rest in their equilibrium positions, subject to harmonic forcing at the end described by $\psi (t) = A \sin (\Omega t)$ at any time $t$. The functions $\phi$ and $\varphi$ are both identically equal to zero and, in order to avoid the shock wave produced by the vanishing initial velocity in the boundary, we let the driving amplitude linearly increase from $0$ to $A$ in a finite period of time before we start to compute the total energy. In the present section we will let $\beta$ be equal to zero and consider a discrete system of $200$ coupled oscillators described by (\ref{Eqn:DiscreteMain}) over a typical time period of $200$, with a time step of $0.05$ and a coupling coefficient equal to $4$.

Let us first examine the massless case when no damping is present at all and $\Omega = 0.9$. In order to verify that our mixed-value problem produces a bifurcation it is necessary to study the qualitative behavior of solutions near the predicted threshold value $A _s$ which, in this case, reads approximately $1.80$. The first row of Figure \ref{Paper2Fig1} shows the function $u _{60} (t)$ in the solution of a sine-Gordon system for two different values of the amplitude of the driving end, while the second row presents the corresponding graphs in the solution of a Klein-Gordon chain. The graphs evidence the presence of a bifurcation in the behavior of solutions around the amplitude value $1.79$ for both chains. 

Naturally, the next step in our investigation will be to determine the behavior of the total energy flow injected by the periodic forcing at the end of the undamped discrete chain of oscillators as a function of the amplitude. The solid lines of Figure \ref{Paper2Fig2} represent the graphs of total energy transmitted into the system vs. amplitude for a forcing frequency of $0.9$ and external damping coefficient equal to zero, for a sine-Gordon system in the first column and a Klein-Gordon system in the second. It is worthwhile noticing the abrupt increase in the total energy administered to the system for some amplitude value between $1.77$ and $1.79$. This feature of the graphs evidences the existence of a bifurcation value around the predicted amplitude $A _s$, after which the phenomenon of nonlinear supratransmission takes place.

Figure \ref{Paper2Fig2} also presents graphs of total energy vs. forcing amplitude for weak constant damping coefficients $\gamma = 0.01$, $0.02$ and $0.03$ in a sine-Gordon system of oscillators. The graphs show a tendency of the bifurcation value to increase linearly as the external damping coefficient is increased. Another interesting feature of this figure is the decrease of total energy for increasing values of $\gamma$, at least for fixed amplitudes greater than the undamped bifurcation threshold. Needless to point out that similar conclusions are obtained for Klein-Gordon chains of oscillators. 

The qualitative effect of $m$ in a sine-Gordon system is also of interest in the analysis of solutions of this chain and is numerically carried out in Figure \ref{Paper2Fig3} for a chain with external damping equal to $0.01$ and pure-imaginary and real masses, using graphs of total energy administered into the system through the boundary vs. driving amplitude. A horizontal shift in the occurrence of the bifurcation value is readily noticed in these graphs. Indeed, the displacement of the bifurcation amplitude for the system (\ref {Eqn:DiscreteMain}) of mass $m$ with respect to the bifurcation amplitude of the massless system is a monotone increasing function of $m ^2$. Analogous computational results (not included here) were obtained for a similar Klein-Gordon chain.

\subsection{Internal damping}

Consider again a system of $200$ oscillators ruled by mixed-value problem (\ref {Eqn:DiscreteMain}) over a time period of $200$, with time step $0.05$, coupling coefficient equal to $4$ and constant external damping equal to zero. In this context, Figure \ref{Paper2Fig3half} shows graphs of total energy vs. forcing amplitude for internal damping coefficients $\beta = 0$, $0.1$, $0.2$ and $0.3$, for sine-Gordon and Klein-Gordon systems. As in the case of external damping, we observe that the threshold value at which supratransmission starts tends to increase as the value of $\beta$ is increased. Opposite to the case of external damping, though, the minimum value for which supratransmission starts varies with $\beta$ in a nonlinear way. 

Next we verify our conclusions on the effect of the mass $m$ on the qualitative behavior of solutions of the sine-Gordon system. Figure \ref{Paper2Fig4} shows graphs of total energy vs. driving amplitude for a system with no external damping, $\beta = 0.2$ and driving frequency of $0.9$, for different pure-imaginary and real masses. As observed before, the graphs evidence a shift on the bifurcation amplitude with respect to the massless bifurcation value which is an increasing function of $m ^2$.

\subsection{Bifurcation analysis}

\begin{figure*}[tcb]
\centerline{
\begin{tabular}{cc}
\scriptsize{Pure-imaginary masses} &
\scriptsize{Real masses} \\
\includegraphics[width=0.45\textwidth]{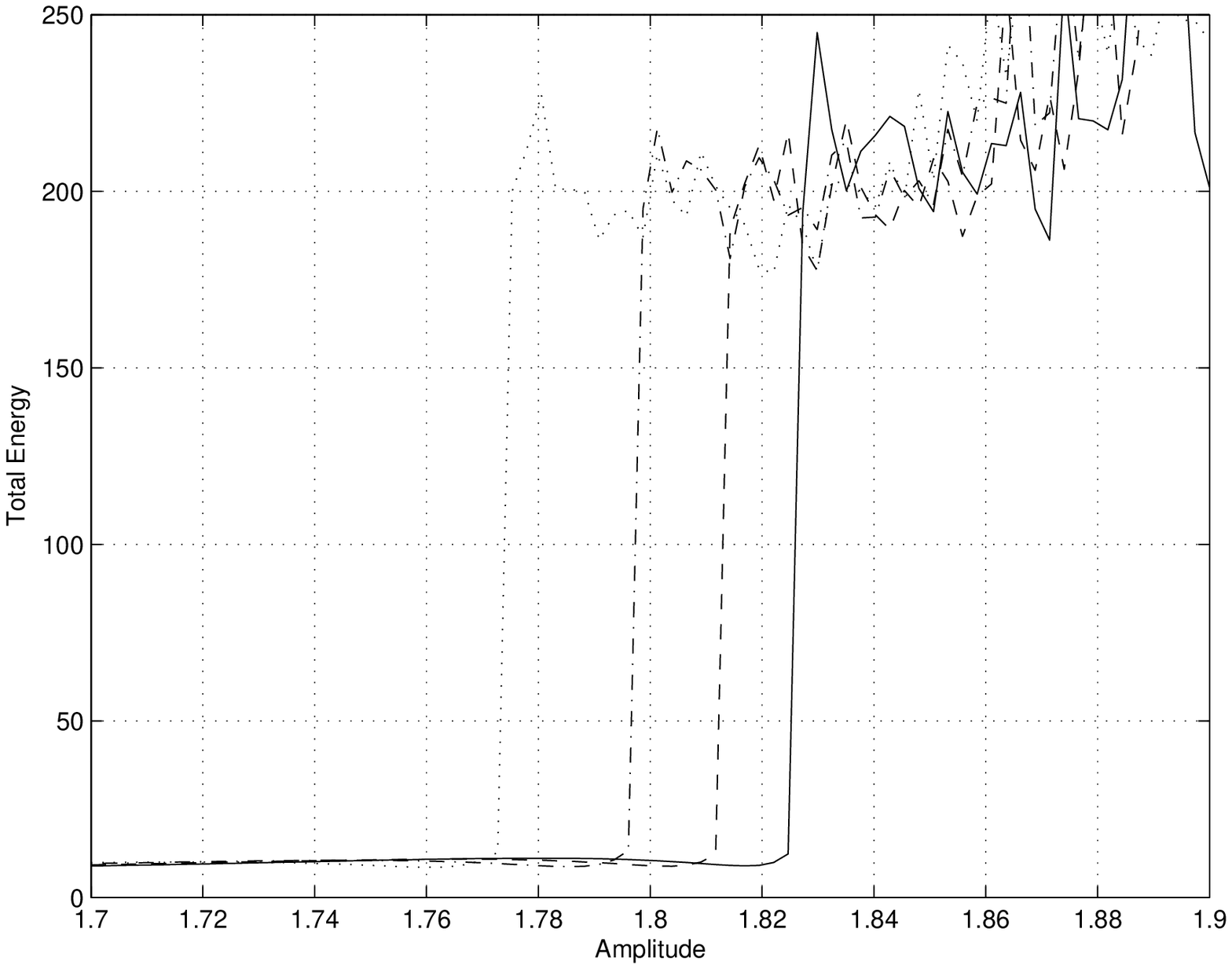} &
\includegraphics[width=0.45\textwidth]{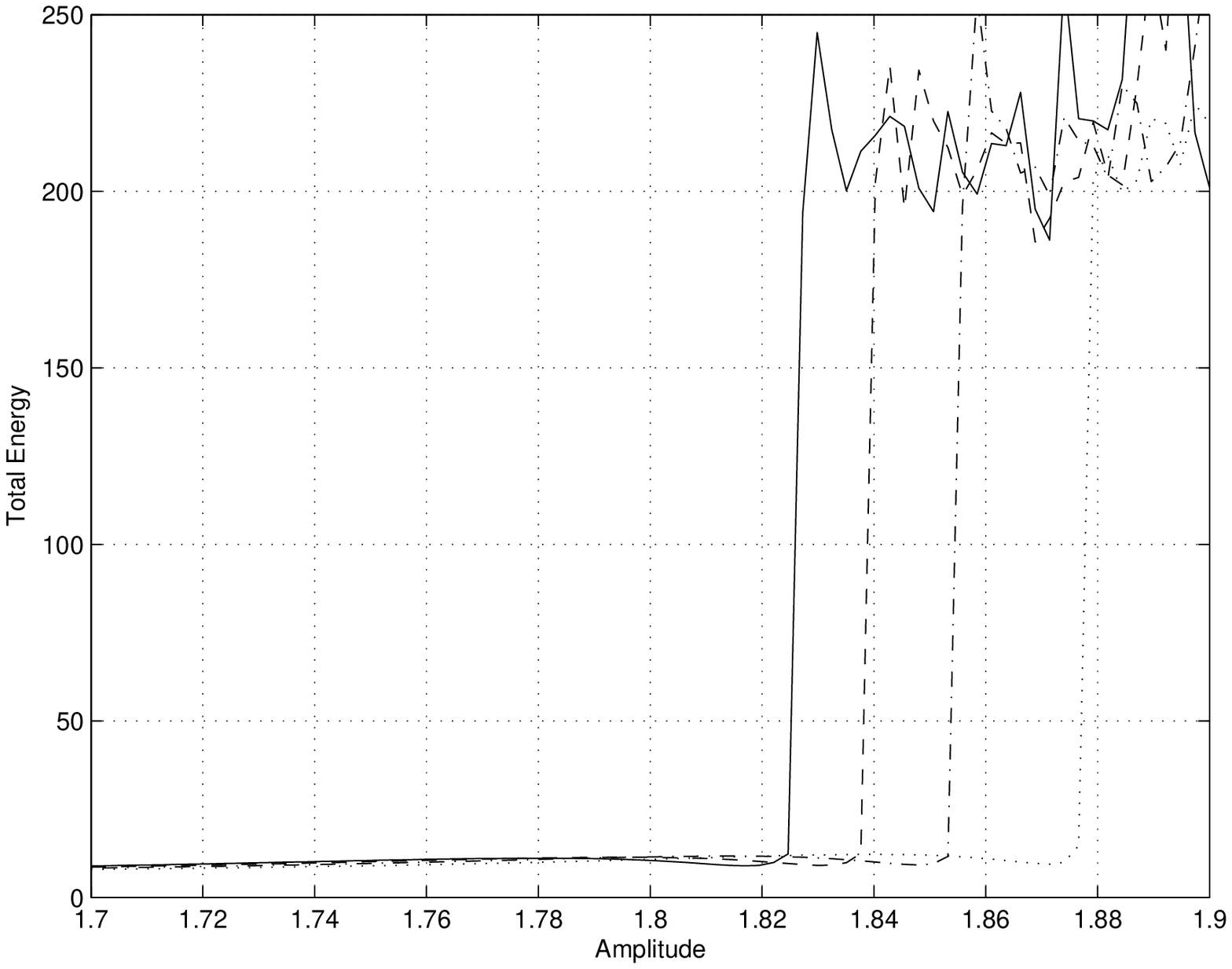}
\end{tabular}}
\caption{Total energy transmitted into the sine-Gordon system vs. driving amplitude for a driving frequency of $0.9$, with $\beta = 0.2$, $\gamma = 0$, and pure-imaginary and real masses of magnitude $0$ (solid), $0.05$ (dashed), $0.075$ (dash-dotted) and $0.1$ (dotted). \label{Paper2Fig4}} %
\end{figure*}

Both the chain of coupled sine-Gordon equations and the chain of Klein-Gordon equations have numerically proved to undergo nonlinear supratransmission for a driving frequency equal to $0.9$ and different values of the external and internal damping coefficients, thus establishing that the results obtained in this work do not depend on integrability. Naturally we are interested in determining that the process of nonlinear supratransmission happens for any frequency value in the forbidden band gap. With that purpose in mind, we obtained graphs of total energy vs. driving frequency and various amplitude values for undamped discrete systems of $200$ coupled oscillators with coupling coefficient $4$, time period of $200$ and time step $0.05$. The $3$-dimensional results for both chains of oscillators are shown in Figure \ref{Paper2Fig11} together with a graph of the continuous-limit threshold amplitude $As$ vs. driving frequency on the amplitude-frequency plane for comparison purposes.

\begin{figure*}[tcb]
\centerline{
\begin{tabular}{cc}
\scriptsize{Sine-Gordon chain} &
\scriptsize{Klein-Gordon chain} \\
\includegraphics[width=0.45\textwidth]{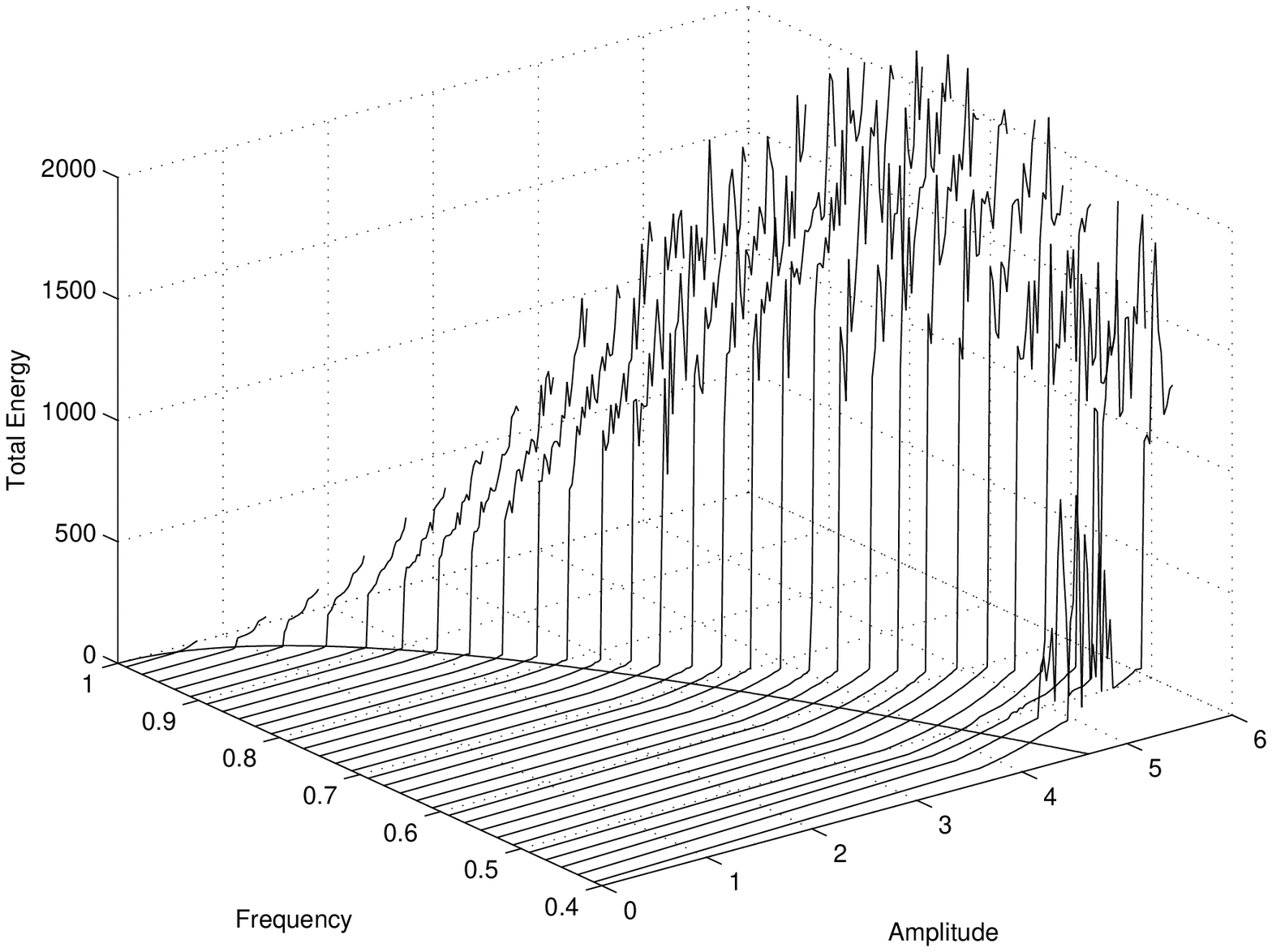} &
\includegraphics[width=0.45\textwidth]{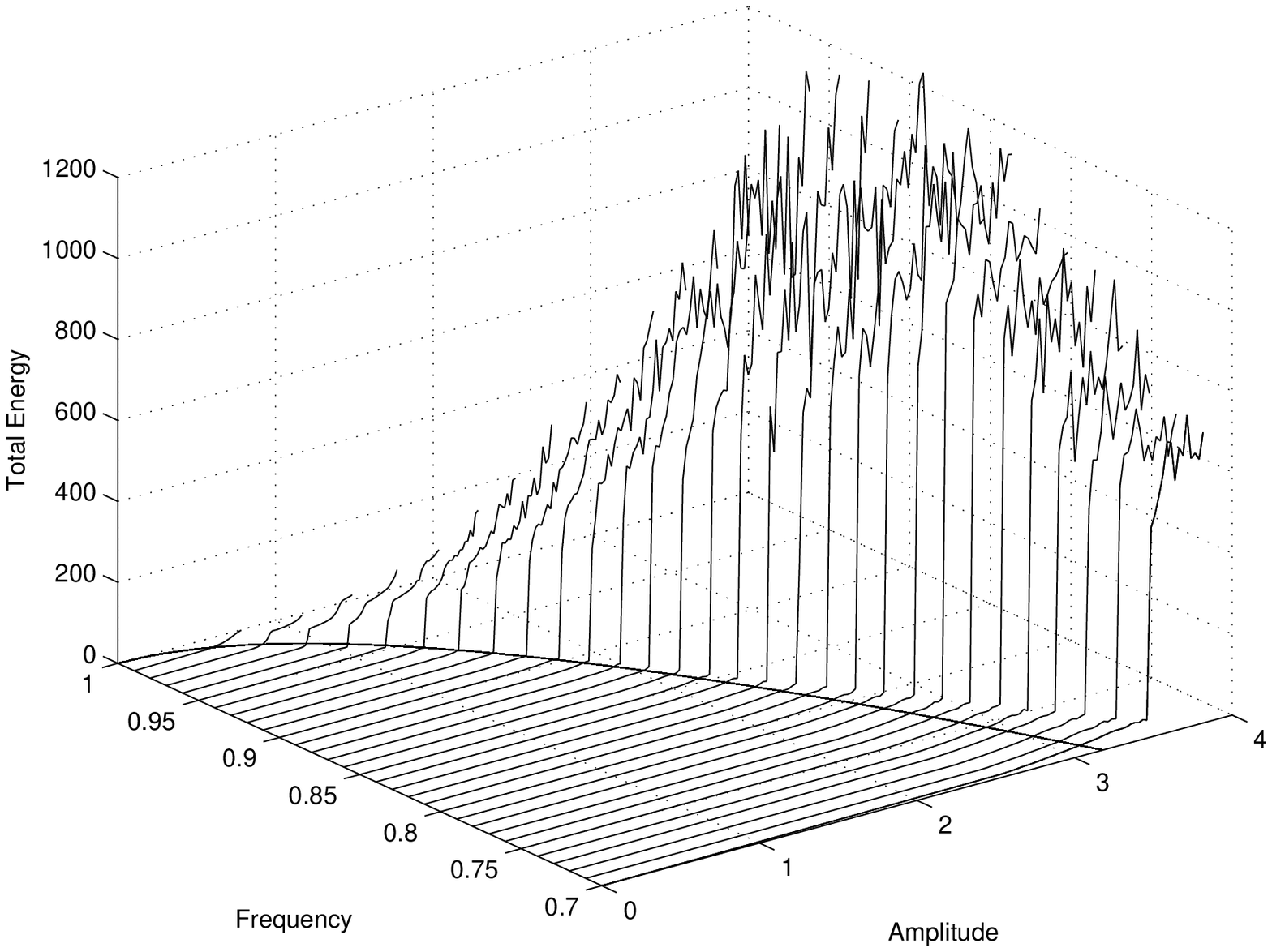}
\end{tabular}}
\caption{Graphs of total energy vs. driving frequency and amplitude for massless and undamped sine-Gordon and Klein-Gordon chains of coupled oscillators.\label{Paper2Fig11}} %
\end{figure*}

Several observations may be immediately drawn from Figure \ref{Paper2Fig11}. First of all, the predicted bifurcation values $As$ display an excellent concordance to the corresponding values obtained using the finite difference-schemes associated with the sine-Gordon and the Klein-Gordon chains. Second, the process of supratransmission ceases to appear in the Klein-Gordon case for driving frequencies below $0.7$. Third, for driving frequencies close to the band gap limit, the bifurcation threshold is not clearly determined from the energy vs. driving amplitude graph of the sine-Gordon chain, as prescribed by \cite {frenchpaper}. For those frequencies, it is indispensable to increase the time period of approximation at least up to $500$. Fourth, strong numerical proof of the existence of the occurrence of the supratransmission process is at hand and we have now reasons to believe that there exists a bifurcation function $A ( \Omega ; \beta , \gamma , m ^2)$ for driving amplitude associated with (\ref {Eqn:DiscreteMain}).

We proceed then to obtain graphs of amplitude values for which nonlinear supratrasmission starts vs. driving frequency for a massless sine-Gordon chain of coupled oscillators with  internal damping coefficient equal to zero and different values of $\gamma$. The numerical results are summarized in Figure \ref{Paper2Bif1} together with the plot of the prescribed continuous-limit bifurcation amplitudes $As$. It is worth noticing that the bifurcation threshold increases with $\gamma$ for fixed frequencies above $0.35$, as previously evidenced for a frequency equal to $0.9$. We must notice also that the discrepancies that appear for frequencies below $0.35$ when $\gamma = 0$, also appear for greater values of $\gamma$, each time bounded in smaller intervals. In fact, we have checked that the discrepancy region --- an effect of the phenomenon of harmonic phonon quenching --- tends to vanish for higher values of external damping (results not included). In this state of matters, we wish to point out that better numerical approximations to the bifurcation threshold are obtained for larger systems of oscillators at the expense of superior needs in terms of computational time. Likewise, we have confirmed that larger values of the coupling coefficient lead to better approximations to the bifurcation threshold, as prescribed by \cite {frenchpaper}.

\begin{figure}[t]
\centerline{
\includegraphics[width=0.6\textwidth]{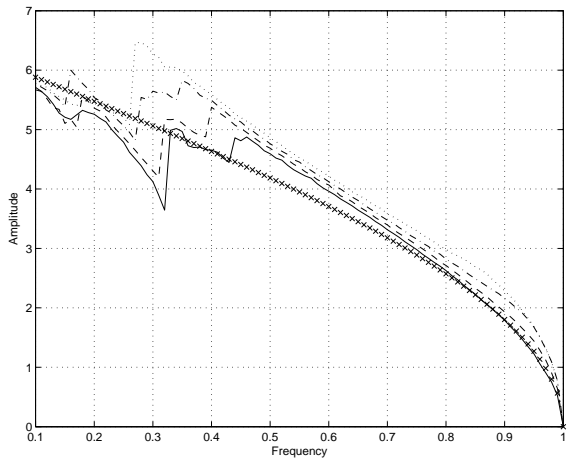}}
\caption{Diagram of smallest driving amplitude value at which supratransmission begins vs. driving frequency for a massless system with $\beta = 0$ and values of $\gamma$ equal to $0$ (solid), $0.1$ (dashed), $0.2$ (dash-dotted) and $0.3$ (dotted). The theoretical threshold $A _s$ in the undamped is shown as a sequence of $\times$-marks. \label{Paper2Bif1}} %
\end{figure}

Figure \ref{Paper2Bif2} shows bifurcation diagrams for a massless sine-Gordon system of oscillators with no external damping and different constant internal damping coefficients. As expected, the bifurcation threshold tends to increase with $\beta$ for a fixed driving frequency. The effects of harmonic phonon quenching are present again in all the bifurcation diagrams, but contrary to the case of external damping, in the case of internal damping the range over which discrepancies occur slightly widens as $\beta$ increases. Also, it is worth observing that the length of the forbidden band gap increases with the parameter $\beta$ apparently in a linear fashion. Moreover, the graphs of bifurcation diagrams for nonzero values of $\beta$ are approximately obtained by shifting horizontally the corresponding graph of the undamped diagram a number of $\beta$ units to the right. More concretely, $A (\Omega; 0, 0 , 0)$ is a good approximation for $A (\Omega - \beta; \beta, 0, 0)$.

\begin{figure}[t]
\centerline{
\includegraphics[width=0.6\textwidth]{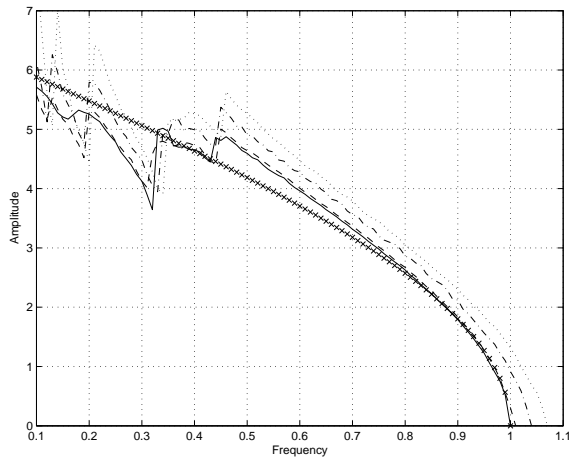}}
\caption{Diagram of smallest driving amplitude value at which supratransmission begins vs. driving frequency for a massless system with $\gamma = 0$ and values of $\beta$ equal to $0$ (solid), $0.1$ (dashed), $0.4$ (dash-dotted) and $0.7$ (dotted). The theoretical threshold $A _s$ in the undamped case is shown as a sequence of $\times$-marks. \label{Paper2Bif2}} %
\end{figure}

Finally, we compute bifurcation diagrams for an undamped sine-Gordon system with different real and pure-imaginary masses in order to establish the effect of $m$ on the occurrence of the bifurcation threshold. The numerical results are summarized in Figure \ref {Paper2Bif3} for some real and pure-imaginary masses, and driving frequencies starting at $0.5$. Our results lead us to conclude that the bifurcation diagram of the sine-Gordon system of oscillators with mass $m$ is approximately equal to the graph of the corresponding massless graph shifted $\sqrt {1 + m ^2} - 1$ horizontal units, for $| m | \ll 1$. In order to test our claim numerically, we obtained graphs of absolute differences between the massless undamped bifurcation diagram, and shifted undamped bifurcation diagrams for several mass values (results not included). We observed that the differences tend to attenuate for high frequency values and that smaller differences are obtained for smaller values of $m$ in magnitude.

\section{Conclusions}

In this paper we have developed a numerical method to approximate solutions of a semi-infinite nonlinear chain of coupled oscillators ruled by modified sine-Gordon equations harmonically driven at its end. The proposed finite-difference scheme is consistent order $\mathcal {O} (\Delta t) ^2$ and we provided a necessary condition in order for the method to be stable order $n$. The process of nonlinear supratransmission for a coupled system of oscillators described by sine-Gordon equations was studied numerically under the scope of this numerical technique, and the dependence of supratransmission on damping was analyzed.

\begin{figure}[tbc]
\centerline{
\includegraphics[width=0.6\textwidth]{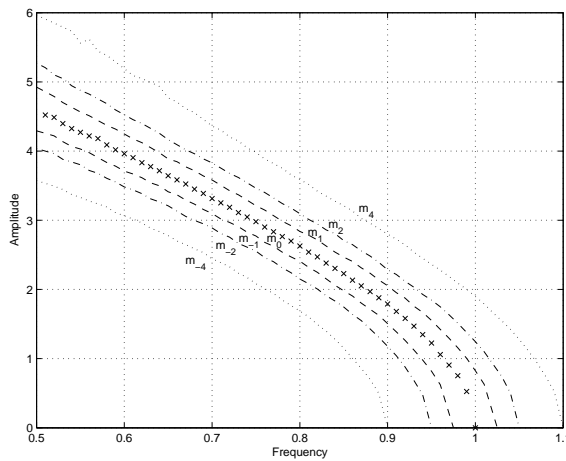}}
\caption{Diagram of smallest driving amplitude value at which supratransmission begins vs. driving frequency for an undamped sine-Gordon system with $\sqrt {m _\ell ^2 + 1} = 1 + \frac {\ell} {40}$ for $\ell = -4$, $-2$, $-1$, $0$, $1$, $2$, $4$. \label{Paper2Bif3}} %
\end{figure}

Several conclusions can be drawn from our computational experiments on the sine-Gordon system of coupled oscillators (\ref {Eqn:DiscreteMain}). First of all, we have shown that the phenomenon of harmonic phonon quenching still appears in the presence of external and internal damping and that the discrepancy region due to phonon quenching is shortened as the external damping coefficient is increased, while it slightly widens as the internal damping coefficient increases. Second, the threshold value at which supratransmission first occurs for fixed frequencies outside the discrepancy region is seen to increase for both external and internal damping as the damping coefficient increases; both conclusions are clearly consequences of the dispersive and dissipative natures of the parameters $\beta$ and $\gamma$, respectively. 

Finally, the bifurcation diagram for a value of the parameter $\beta$ is approximately equal to the corresponding diagram for the undamped system shifted $\beta$ horizontal units to the right. Likewise, a horizontal shift of $\sqrt {1 + m ^2} - 1$ units in the bifurcation diagram of a sine-Gordon system of mass $m$ with respect to the corresponding massless system is observed for small masses and frequencies outside the discrepancy region.

\subsubsection*{Acknowledgment}

It is our duty to express our deepest gratitude to Dr. \'{A}lvarez Rodr\'{\i}guez, dean of the Centro de Ciencias B\'{a}sicas of the Universidad Aut\'{o}noma de Aguascalientes, and to Dr. Avelar Gonz\'{a}lez, head of the Direcci\'{o}n General de Investigaci\'{o}n y Posgrado of the same university, for uninterestedly providing us with computational resources to produce this article. The present work represents a set of partial results under project PIM07-2 at this university.

\end{document}